\documentclass[final,110pt,3p]{article}
\usepackage{setspace}
\usepackage{amsfonts}
\usepackage{amssymb}
\usepackage{amsthm}
\usepackage{amsmath}
\usepackage{algorithmic}
\usepackage{color}
\usepackage{graphicx}
\usepackage{verbatim}
\usepackage[round,colon,authoryear]{natbib}
\usepackage{subfigure}

\newtheorem{theorem}{Theorem}

\newtheorem{corollary}{Corollary}

\newtheorem{lemma}{Lemma}

\newtheorem{proposition}{Proposition}
\newtheorem{remark}{Remark}
\theoremstyle{remark}

\theoremstyle{remark}

\newcommand{\1}{\mathbf{1}}

\newcommand{\N}{\mathbb{N}}

\newcommand{\dd}{\mathrm{d}}
\begin{document}

\title{A Weak Convergence Criterion Constructing Changes of Measure\thanks{%
We thank Richard Davis, Paul Embrechts, and Thomas Mikosch for putting
together a very interesting Oberwolfach seminar, where this project started.
We are grateful to Kay Giesecke, Peter Glynn, Jan Kallsen, Ioannis Karatzas,
and Philip Protter for stimulating conversations on topics related to the
theme of this project, and to Philippe Charmoy, Zhenyu Cui, Roseline
Falafala, and Nicolas Perkowski for their comments on an earlier version of
this note.}}
\author{Jose Blanchet\thanks{%
E-Mail: jose.blanchet@columbia.edu} \\
Department of Operations Research and Industrial Engineering\\
Columbia University \and Johannes Ruf\thanks{%
E-Mail: johannes.ruf@oxford-man.ox.ac.uk} \\
Oxford-Man Institute of Quantitative Finance and Mathematical Institute\\
University of Oxford}
\maketitle

\begin{abstract}
Based on a weak convergence argument, we provide a necessary and sufficient
condition that guarantees that a nonnegative local martingale is indeed a
martingale. Typically, conditions of this sort are expressed in terms of
integrability conditions (such as the well-known Novikov condition). The
weak convergence approach that we propose allows to replace integrability
conditions by a suitable tightness condition. We then provide several
applications of this approach ranging from simplified proofs of classical
results to characterizations of processes conditioned on first passage time
events and changes of measures for jump processes.
\end{abstract}



\section{Introduction}

Changing the probability measure is a powerful tool in modern probability.
Changes of measure arise in areas of wide applicability such as in
mathematical finance, in the setting of so-called equivalent pricing
measures. A change of probability measure often relies on the specification
of a nonnegative martingale process which in turn yields the underlying
Radon-Nikodym derivative behind the change of measure.

The key step in the typical construction of changes of measure involves
showing the martingale property of a process of putative Radon-Nikodym
derivatives. In order to verify this martingale property one often starts by
defining a process that easily can be seen to be a local martingale. This is
the standard situation, for example, in changes of measure for diffusion
processes; in this framework, a standard application of It\^{o}'s formula
guarantees that a candidate exponential process is a local martingale. The
difficult part then involves ensuring that the local martingale is actually
a martingale.

Since the distinction between local martingales and martingales involves
verification of integrability properties (the ones behind the strict
definition of a martingale), it is most natural to search for a criterion
based on integrability of the underlying local martingale. This is the
basis, for instance, of the so-called Novikov's condition, which is a
well-known criterion used to verify the martingale property of an
exponential local martingale in the diffusion setting. Nevertheless, if
ultimately one has the existence of a new probability measure, then one has
a martingale defined by the corresponding change of measure. Thus, it
appears that lifting the local martingale property for a nonnegative
stochastic process to a bona-fide martingale property has more to do with
the fact that the induced probability measure is indeed well-defined.

Our contribution in this note consists in putting into focus the aspect of
tightness when proving the martingale property of a nonnegative local
martingale. Connecting tightness with the verification of the martingale
property is an almost trivial exercise, formulated in Theorem~\ref{T
mainmain} below. Although only a very simple observation, this point of view
is powerful as the applications in Section~\ref{S ex} illustrate. In
particular, we illustrate our result in the context of the following four
applications:

\begin{enumerate}
\item We provide a new proof of the result by \citet{Benes_1971} on the
existence of weak solutions to certain stochastic differential equations.

\item We prove the equivalence of weak solutions to stochastic differential equations that involve compound Poisson processes, whose
intensity may depend on the current state of the system.

\item We weaken the assumptions of \citet{Giesecke_2013} that yield the
martingale property of certain local martingales involving counting processes.

\item We provide a new representation for conditional expectations of an Ornstein-Uhlenbeck process conditioned to hit a large
level before hitting zero. We believe that this representation is useful for simulation purposes.
\end{enumerate}

For the sake of clear notation, for a sequence of random variables $%
\{Y_{n}\}_{n\in\mathbb{N}}$, each defined on a probability space $\left(
\Omega _{n},\mathcal{F}_{n},P_{n}\right) $, and a random variable $Y$,
defined on a probability space $\left( \Omega,\mathcal{F},P\right) $, we
write
\begin{equation*}
(P_{n},Y_{n})\overset{\mathfrak{w}}{\Longrightarrow}(P,Y) \text{ } (n
\uparrow \infty) \qquad \text{if} \qquad \lim_{n \uparrow \infty}
P_{n}\left( Y_{n}\leq x\right) = P\left( Y\leq x\right) \text{ for each
continuity point $x$ of $P\left( Y\leq \cdot\right)$}.
\end{equation*}

The proof of the following Theorem~\ref{T mainmain} is very simple and only
relies on the definition of tightness; it is given in Section~\ref{S main}.

\begin{theorem}
The following two statements hold: \label{T mainmain}

\begin{enumerate}
\item Let $M=\{M\left( t\right) \}_{t\geq0}$ denote a nonnegative sub- or
supermartingale on a filtered probability space $\left( \Omega,\mathcal{F},
\{\mathcal{F}(t)\}_{t \geq0}, P\right) $ with corresponding expectation
operator $E$, and let $\{M_{n}\}_{n \in\mathbb{N}}$ with $%
M_{n}=\{M_{n}\left( t\right) \}_{t\geq0}$ denote a sequence of nonnegative
martingales, each defined on a filtered probability space $\left( \Omega_{n},%
\mathcal{F}_{n}, \{\mathcal{F}_{n}(t)\}_{t \geq0}, P_{n}\right) $ with
corresponding expectation operators $E_{n}$ such that $M_{n}(0) = 1$. Fix
any sequence of (deterministic) times $\{t_{m}\}_{m\in \mathbb{N}}$ with $%
t_{1}=0$ and $\lim_{m\uparrow \infty }t_{m}=\infty $ and assume that $%
(P_{n},M_{n}(t_{m}))\overset{\mathfrak{w}}{\Longrightarrow }(P,M(t_{m}))$ $%
(n \uparrow \infty)$ for each $m\in \mathbb{N}$. Define a family $%
\{Q_{n}^{m}\}_{n,m \in\mathbb{N}}$ of probability measures via $\mathrm{d}%
Q_{n}^{m}=M_{n}(t_{m})\mathrm{d}P_{n}$.

Then $M$ is a true martingale with $M(0)=1$ if and only if
\begin{align}  \label{E Q tightness}
\sup_{n\in \mathbb{N}}Q_{n}^{m}(M_{n}(t_{m})\geq \kappa )\rightarrow 0
\end{align}%
as $\kappa \uparrow \infty $ for each $m\in \mathbb{N}$. That is, $M$ is a
true martingale if and only if $\{M_{n}(t_{m})\}_{n\in \mathbb{N}}$ is tight
under the sequence of measures $\{Q_{n}^{m}\}_{n\in \mathbb{N}}$ for each $%
m\in \mathbb{N}$.

\item Let $M(\infty)$ denote a nonnegative random variable on a probability
space $\left( \Omega,\mathcal{F}, P\right) $ with corresponding expectation
operator $E$, and let $\{M_{n}(\infty)\}_{n \in\mathbb{N}}$ denote a
sequence of nonnegative random variables, each defined on a probability
space $\left( \Omega_{n},\mathcal{F}_{n}, P_{n}\right) $ with corresponding
expectation operators $E_{n}$ such that $E_n[M_{n}(\infty)] = 1$. Assume
that $(P_{n},M_{n}(\infty))\overset{\mathfrak{w}}{\Longrightarrow }%
(P,M(\infty))$ $(n \uparrow \infty)$. Define a family $\{Q_{n}\}_{n \in%
\mathbb{N}}$ of probability measures via $\mathrm{d}Q_{n}=M_{n}(\infty)%
\mathrm{d}P_{n}$.

Then ${\mathbb{E}}[M(\infty)]=1$ holds if and only if
\begin{equation*}
\sup_{n\in \mathbb{N}}Q_{n}(M_{n}(\infty )\geq \kappa )\rightarrow 0
\end{equation*}%
 as $\kappa \uparrow \infty $.
\end{enumerate}
\end{theorem}

It is important to note that showing the martingale property of the
underlying positive local martingale becomes an exercise in tightness in a
very weak topology. Given the enormous literature on weak convergence
analysis of stochastic processes, we feel that our test of martingality
would be a useful one. For example, in order to show tightness of a sequence
of random variables $\{A_{n}\}_{n\in \mathbb{N}}$ of the form $A_{n}=\exp
(B_{n}+C_{n})$ it is sufficient to show tightness for the sequences of
random variables $\{B_{n}\}_{n\in \mathbb{N}}$ and $\{C_{n}\}_{n\in \mathbb{N%
}}$ separately; a task that is often easy, as we shall illustrate in Section~\ref{S ex}. In
addition, the martingale property of a natural approximation to the local
martingale process of interest is usually immediately seen to be a
martingale.

\subsection*{Relevant literature}

The standard way to show the martingale property of a nonnegative local
martingale is to check some standard integrability condition; see for
example \citet{Novikov}, \citet{Kazamaki_1983}, or \citet{Ruf_Novikov}. If
the local martingale dynamics include jumps, a case that we explicitly allow
here, then integrability conditions exist but they might not be trivial to
check; see \citet{Lepingle_Memin_Sur} and \citet{Protter_Shimbo} for such
conditions and related literature.

Under additional assumptions on the local martingale, such as the assumption
that it is constructed via an underlying Markovian process, further
sufficient (and sometimes also necessary) criteria can be derived. Here we
only provide the reader with some pointers to this vast literature. The
following papers develop conditions different from Novikov-type conditions
by utilizing the (assumed) Markovian structure of some underlying stochastic
process, and contain a far more complete list of references: \citet{CFY}, %
\citet{Blei_Engelbert_2009}, \citet{MU_martingale}, and %
\citet{Ruf_martingale}. \citet{KMK2010} study the martingale property of
stochastic exponentials of affine processes; their approach via the explicit
construction of a candidate measure and the use of a simple lemma in %
\citet{JacodS} is close in spirit to our approach.

The weak existence of solutions to stochastic differential equations is
often proven by means of changing the probability measure, see for example %
\citet{Portenko_1975}, \citet{Engelbert_Schmidt_1984}, \citet{Yan_1988}, or %
\citet{Stummer_1993}. This strategy for proving the weak existence of
solutions requires the true martingale property of the putative
Radon-Nikodym density. Our approach to prove the martingale property of such
a density is in the spirit of the reverse direction: The tightness condition
that implies the martingale property of a putative Radon-Nikodym density by
Theorem~\ref{T mainmain} corresponds basically to the asserted existence of
a certain probability measure --- often corresponding to the existence of a
solution to a stochastic differential equation.

\section{Martingale property and tightness}

\label{S main} In this section, we prove Theorem~\ref{T mainmain} and make
some related observations. The proof of Theorem~\ref{T mainmain} relies on
the following simple but powerful result:

\begin{proposition}
\label{T main} Let $Y$ denote a nonnegative random variable defined on $%
\left( \Omega,\mathcal{F}, P\right) $ with corresponding expectation
operator $E$, and let $\{Y_{n}\}_{n \in\mathbb{N}}$ denote a sequence of
integrable, nonnegative random variables defined on $\left( \Omega _{n},%
\mathcal{F}_{n}, P_{n}\right) $ with corresponding expectation operators $%
E_{n}$ such that $(P_{n},Y_{n} )\overset{\mathfrak{w}}{\Longrightarrow}(P,Y
) $ $(n \uparrow \infty)$ and $\lim_{n \uparrow\infty} E_{n}[Y_{n}] = 1$.
Then, $E[Y] = 1$ holds if and only if
\begin{equation*}
\sup_{n \in\mathbb{N}} E_{n}\left[ Y_{n} \mathbf{1}_{\{Y_{n} \geq\kappa \}}%
\right] \rightarrow0
\end{equation*}
does as $\kappa\uparrow\infty$.
\end{proposition}

\begin{proof}
Assume that $E[Y] = 1$. Then, for fixed $\kappa>1$ and for a continuous
function $f: [0,\infty] \rightarrow[0,\kappa]$ with $f(x) \leq x$ for all $x
\geq 0$, $f(x) = x$ for all $x \in[0, \kappa-1]$ and $f(x) = 0 $ for all $x
\in[\kappa, \infty)$, we compute that
\begin{align*}
E_n[Y_n \mathbf{1}_{\{Y_n \geq \kappa\}}] &= E_n[Y_n] - E_{n}[Y_n \mathbf{1}%
_{\{Y_{n} <\kappa\}} ] \leq E_n[Y_n] -E_{n}[f(Y_n) ] \\
& \rightarrow1 - E[f( Y) ] \leq1 - E[Y \mathbf{1}_{\{Y \leq\kappa- 1\}}] =
E[Y \mathbf{1}_{\{Y > \kappa- 1\}} ]
\end{align*}
as $n \uparrow\infty$. As $E[Y \mathbf{1}_{\{Y > \kappa- 1\}} ]$ can be made
arbitrarily small by increasing $\kappa$ (because $Y$ is integrable by
assumption), we obtain one direction of the statement. For the other
direction, fix $\epsilon>0$ and the continuous, bounded function $f:
[0,\infty] \rightarrow {\mathbb{R}}$ with $f(x) =x \wedge \kappa$ for all $x
\geq 0$. Then
\begin{align*}
E[Y] \geq E[f(Y)] = \lim_{n \uparrow \infty} E_n[f(Y_n)] \geq \liminf_{n
\uparrow \infty} E_n[Y_n \mathbf{1}_{\{Y_n < \kappa\}}] = \liminf_{n
\uparrow \infty} \left(E_n[Y_n] - E_n[Y_n \mathbf{1}_{\{Y_n \geq \kappa\}}]
\right) \geq 1-\epsilon
\end{align*}
for $\kappa$ large enough. This yields $E[Y] \geq 1$. Similarly, we can show
that $E[Y] \leq 1$.
\end{proof}

We are now ready to prove Theorem~\ref{T mainmain}:

\begin{proof}[Proof of Theorem~\protect\ref{T mainmain}] The second statement is a (slightly weakened) reformulation of Proposition~\ref{T main}.
For the first statement, observe that \eqref{E Q tightness} and the
martingale property of all processes $M_{n}$ imply that $E[M(t_{m})] = 1$
for all $m \in\mathbb{N}$ by Proposition~\ref{T main}. Since $E[M(t)]$ is
assumed to be monotone in $t$, this yields the martingale property of $M$.
The reverse direction is a direct application of the same proposition.
\end{proof}

A look at its proof yields that the statement of Theorem~\ref{T mainmain}
can be further generalized since for each $t_{m}$ a different approximating
sequence of true martingales might be used.

The following corollary can be interpreted as a generalization of
Theorem~1.3.5 in \citet{SV_multi} to processes with jumps. See also
Lemma~III.3.3 in \citet{JacodS} for a similar statement where a certain
candidate measure $Q$ is assumed to exist. We remark that the sequence of
stopping times in the statement could, but need not, be a localization
sequence of a local martingale; for example, it is sufficient that the
stopping times converge to the first hitting time that the underlying local
martingale hits zero.

\begin{corollary}
\label{cor stopping times} Let $\{\tau _{n}\}_{n\in \mathbb{N}}$ be a
sequence of stopping times and $M_{n}\equiv M^{\tau _{n}}$ the stopped
versions of a nonnegative local martingale $M$ with $M(0)=1$. Assume that $%
M_{n}(t)\rightarrow M(t)$ $P$-a.s.~as $n\uparrow \infty $ for all $t>0$ and
that, for each fixed $n$, $M_{n}$ is a uniformly integrable martingale.
Further, define $\mathrm{d}Q_{n}=M_{n}(\infty )\mathrm{d}P$. Under these
assumptions the following statements hold: If $Q_{n}(\tau _{n}\leq
t)\rightarrow 0$ as $n\uparrow \infty $ for all $t>0$, then $M$ is a
martingale. Further, under the additional assumption that $\tau
_{n}\rightarrow \infty $ $P$-a.s. as $n\uparrow \infty $, the converse also
holds; that is, if $M$ is a martingale then $Q_{n}(\tau _{n}\leq
t)\rightarrow 0$ as $n\uparrow \infty $ for all $t>0$.
\end{corollary}

\begin{proof}
Fix $t$ and $\kappa>0$ and observe that $(P,M_{n}\left( t\right) )\overset{%
\mathfrak{w}}{\Longrightarrow}(P,M\left( t\right) )$  $(n\uparrow\infty)$.
Also note that for each $n\in \mathbb{N}$,
\begin{align*}
Q_{n}\left( M_{n}\left( t\right) >\kappa\right) &\leq Q_n\left(\tau_n\leq
t\right) + E\left[M_n(t) \mathbf{1}_{\{\tau_n > t\} \bigcap \{ M_{n}\left(
t\right) >\kappa\}}\right] \leq Q_n\left(\tau_n \leq t\right) + E\left[M(t)
\mathbf{1}_{ \{ M\left( t\right) >\kappa\}}\right]
\end{align*}
because $M_n(t) \mathbf{1}_{\{\tau_n > t\}} = M(t) \mathbf{1}_{\{\tau_n >
t\}}$. Since by assumption we can make the first term on the right-hand side
arbitrarily small by increasing $n$, the martingale property of $M$ follows
directly from dominated convergence and Theorem~\ref{T mainmain}. For the
reverse direction, assume that $M$ is a martingale and that $\tau_n
\rightarrow \infty$ $P$-a.s. as $n \uparrow \infty$. Then, $Q_n(\tau_n \leq
t) = E[M_n(t) \mathbf{1}_{\{\tau_n \leq t\}}] = E[M(t) \mathbf{1}_{\{\tau_n
\leq t\}}]\rightarrow 0 $ as $n\uparrow\infty$ by dominated convergence.
\end{proof}

The next result is of course well-known and only a very special case of, for
instance, the theory of BMO martingales; see for example %
\citet{Kazamaki_1994}. However, as we shall use the result below and as we
would like to make this note self-contained, we provide a proof based on the
observations we have made here before:

\begin{corollary}
\label{C bounded} Let $L=\{L\left( t\right) \}_{t\geq 0}$ denote a
continuous local martingale on some probability space. Assume there exists
some nondecreasing (deterministic) function $c:[0,\infty )\rightarrow {%
\mathbb{R}}$ such that $\min \{L_{t},\langle L\rangle _{t}\}\leq c(t)$ for
all $t\geq 0$ almost surely. Then, $M=\mathcal{E}(L):=\exp (L-\langle
L\rangle /2)$ is a martingale.
\end{corollary}

\begin{proof}
For each $n \in\mathbb{N}$ let $\tau_{n}$ denote the first hitting times to
level $n$ or higher by $M$ and fix $t > 0$. Obviously, $M_{n}\equiv
M^{\tau_{n}}$ satisfies $M_{n}(t) \rightarrow M(t)$ $P$-a.s.~as $n \uparrow
\infty$. Define the probability measures $Q_{n}$ as in Corollary~\ref{cor
stopping times} and observe that $\{Q_n(\tau_n \leq t)\}_{n \in \mathbb{N}}$
is a decreasing sequence since
\begin{align*}
Q_{n+1}(\tau_{n+1} \leq t) \leq Q_{n+1}(\tau_n \leq t) = Q_n(\tau_n \leq t)
\end{align*}
for all $n \in \mathbb{N}$.

Now, fix $\epsilon \in (0,1)$ and some $m \in \mathbb{N}$ with $m >
\exp(c(t)) / \epsilon$ and observe that
\begin{equation*}
\{M^{\tau_{m}}(t) \geq m\} \subset \{L(t \wedge \tau_m) > c(t)\} \subset \{
\langle L \rangle(t \wedge \tau_m) \leq c(t) \}
\end{equation*}
holds $P$-a.s. Therefore, we have
\begin{align*}
\{\tau_m \leq t\} &= \{M^{\tau_{m}}(t) \geq m\} = \{M^{\tau_{m}}(t) \geq m\}
\cap \{ \langle L \rangle(t \wedge \tau_m) \leq c(t) \} \subset \left\{%
\widetilde{M}^{\tau_{m}}(t) > \frac{1}{\epsilon}\right\}
\end{align*}
$P$-a.s.~and thus $Q_{m}$-a.s., where $\widetilde{M} :=
M^{\tau_m}/\exp(\langle L \rangle(\tau_{m} \wedge \cdot))$ is a bounded,
nonnegative $Q_{m}$--martingale by Girsanov's theorem. Markov's inequality
then implies that $Q_{m}(\tau_{m} \leq t) \leq \epsilon$  and an application
of Corollary~\ref{cor stopping times} concludes.
\end{proof}

The next observation is useful when applying Theorem~\ref{T mainmain} in a
continuous setup:

\begin{lemma}
\label{L quadVar} Assume the notation of the first part of Theorem~\ref{T
mainmain}. Let $\{L_{n}\}_{n\in \mathbb{N}}$ denote a sequence of continuous
$Q_{n}$-local martingales with quadratic variation $\langle L_{n}\rangle $
and assume that the sequence $\{\langle L_{n}\rangle (t)\}_{n\in \mathbb{N}}$
is tight along the sequence $\{Q_{n}\}_{n\in \mathbb{N}}$ of probability
measures for some $t\in \lbrack 0,\infty ]$. Then also the sequence $%
\{L_{n}(t)\}_{n\in \mathbb{N}}$ is tight along $\{Q_{n}\}_{n\in \mathbb{N}}$.
\end{lemma}

\begin{proof}
Fix $n \in \mathbb{N}$, let $\rho_{\kappa}$ denote the first hitting time to
level $\kappa$ or higher by $\langle L_n \rangle$, and observe that
\begin{align*}
Q_{n}\left( L_n(t) > \kappa\right) &\leq Q_{n}\left(L_n(t \wedge
\rho_{\kappa}) > \kappa\right) + Q_{n}\left( \rho_{\kappa} \leq t \right)
\leq \frac{E_{n}\left[L_n^2(t \wedge \rho_{\kappa})\right]}{\kappa^2} +
Q_{n}\left( \rho_{\kappa} \leq t \right) \\
&\leq \frac{1}{\kappa} + Q_{n}\left( \langle L_n\rangle (t) > \kappa\right)
\end{align*}
for all $\kappa>0$ by Chebyshev's inequality and the fact that $E_n[L_n^2(t
\wedge \rho_{\kappa})] \leq E_n[\langle L_n\rangle (t \wedge \rho_{\kappa})
] \leq \kappa$; those last inequalities follow from the observation that the process $L_n^2(\cdot \wedge \rho_{\kappa})
- \langle L_n\rangle (\cdot \wedge \rho_{\kappa})$ is a local martingale, bounded from below by $-\kappa>0$.
\end{proof}

\section{Applications}

\label{S ex}

Our goal here is to show that our approach could have advantages in terms of
its relative simplicity. We shall write $\|\cdot\|$ for the Euclidean $L_2$%
--norm on ${\mathbb{R}}^d$ for some $d \in \mathbb{N}$. We denote the space
of cadlag paths $\omega :[0,t )\rightarrow {\mathbb{R}}^{d}$ for some $d\in
\mathbb{N}$ and $t \in (0,\infty]$,
endowed with the standard Skorokhod topology, by $D_{[0,t )}({\mathbb{R}}^{d})$. For sake of brevity, we shall
use $D_{[0,\infty)} = D_{[0,\infty)}({\mathbb{R}}^1)$.

\subsection{Continuous processes: linear growth of drift}

We begin by proving an extension of the well-known result by %
\citet{Benes_1971} on the existence of weak solutions to a certain
stochastic differential equation. We discuss it to illustrate how
considerations of tightness as suggested here can often simplify the
argument that a certain process is a martingale.

\begin{theorem}
\label{P:Benes} Fix $d \in\mathbb{N}$ and let $W=\{W\left( t\right)
\}_{t\geq0}$ be a $d$-dimensional Brownian motion, $W^{*}=\{W^{*}\left(
t\right) \}_{t\geq0}$ the running maximum of its vector norm; to wit, $%
W^{*}(t) := \max_{s \in[0,t]} \{\|W(s)\|\}, $ and $Y=\{Y\left( t\right)
\}_{t\geq0}$ a nonnegative supermartingale (under the same filtration) with
cadlag paths such that $[Y,W_i]$ is a nonincreasing process for all $i \in
\{1, \ldots, d\}$. Furthermore, let $Y^{*} := \{Y^{*}\left( t\right)
\}_{t\geq0}$ denote its maximum process. Moreover, suppose that $%
\mu=\{\mu\left( t\right) \}_{t\geq0}$ is a progressively measurable process
satisfying
\begin{align*} 
\left\Vert \mu(t) \right\Vert\leq c(t, Y^*(t))\left( 1+ W^{*}(t)\right)
\end{align*}
for all $t \geq 0$ and some function $c:[0,\infty) \times [0,\infty)
\rightarrow [0,\infty)$ that is nondecreasing in both arguments.

Then the local martingale $M=\{M\left( t\right) \}_{t\geq0}$, defined as
\begin{equation*}
M( t) :=\exp\left( \int_{0}^{t}\mu( s ) \mathrm{d} W(s) -\frac{1}{2}%
\int_{0}^{t}\left\Vert \mu( s) \right\Vert^{2}\mathrm{d}s\right) ,
\end{equation*}
is a martingale.
\end{theorem}

\begin{proof}
Let us define the sequence $\{\mu_n\}_{n \in \mathbb{N}}$ of progressively
measurable processes $\mu_{n}=\{\mu _{n}\left( t\right) \}_{t\geq0}$,
defined by $\mu_{n}(\cdot):=(\mu (\cdot)\wedge n)\vee(-n)$, where the
minimum and maximum are taken component by component. It follows easily, for
example by applying the definition of the stochastic integral, that the
sequence of local martingales $M_{n}=\{M_{n}\left( t\right) \}_{t\geq0}$,
defined as
\begin{equation*}
M_{n}( t) :=\exp\left( \int_{0}^{t}\mu_{n}( s) \mathrm{d}W( s) -\frac{1}{2}%
\int_{0}^{t}\left\Vert \mu_{n}( s ) \right\Vert^{2}\mathrm{d}s\right) ,
\end{equation*}
satisfies $(P,M_{n}\left( t\right) )\overset{\mathfrak{w}}{\Longrightarrow}%
(P,M\left( t\right) )$  $(n \uparrow \infty)$ for all $t\geq0$. By Corollary~\ref{C bounded}, the local martingale $M_{n}
$ is a true martingale since $\mu_{n}$ is bounded. Now, observe that
\begin{align*}
B_{n}( \cdot) :=W( \cdot) -\int_{0}^{\cdot}\mu _{n}(s) \mathrm{d} s
\end{align*}
is a Brownian motion and that $Y$ is still a nonnegative supermartingale under the probability measure $Q_{n}$, induced by $%
M_{n}( \cdot) $ via $\mathrm{d} Q_{n} = M_{n}(t) \mathrm{d} P$, and that
\begin{align*}
M_{n}( t) =\exp\left( \int_{0}^{t}\mu_{n}(s) \mathrm{d} B_{n}( s) +\frac{1}{2%
}\int_{0}^{t} \left\Vert \mu _{n}(s) \right\Vert^{2}\mathrm{d} s\right) .
\end{align*}
We first note that%
\begin{align*}
\left\Vert W( t) \right\Vert \leq\left\Vert B_{n}(t)
\right\Vert+\int_{0}^{t}c(s, Y^*(s)) \left( 1+ W^{*}(s)\right) \mathrm{d} s \leq
B_{n}^{*}( t) +c(t, Y^*(t)) t+c(t, Y^*(t)) \int_{0}^{t}W^{*}(s) \mathrm{d} s
\end{align*}
for all $r \leq t$, where $B_{n}^{*}=\{B_{n}^{*}\left( t\right) \}_{t\geq0}$
is defined similar to $W^{*}$. An application of Gronwall's inequality then
yields that $W^{*}(t)$ is tight along $\{Q_n\}_{n \in \mathbb{N}}$. This
guarantees that $\{\int_{0}^{t}\left\Vert \mu_{n}(s) \right\Vert^{2} \mathrm{%
d} s\}_{n \in \mathbb{N}}$ is tight as well. Lemma~\ref{L quadVar} then
yields the tightness of $\{\int_{0}^{t} \mu_{n}(s) \mathrm{d} B_n(s)\}_{n
\in \mathbb{N}}$. Thus, $\{M_n(t)\}_{n \in \mathbb{N}}$ is tight along $%
\{Q_n\}_{n \in \mathbb{N}}$ and $M$ is a true $P$-martingale by Theorem~\ref%
{T mainmain}.
\end{proof}

To recover the result by \citet{Benes_1971}, suppose that $\widetilde{\mu }%
:[0,\infty )\times \mathbb{R}^{d}\rightarrow \mathbb{R}$ is measurable and
satisfies%
\begin{equation*}
\left\Vert \widetilde{\mu }\left( t,x\right) \right\Vert \leq \widetilde{c}%
(t)\left( 1+\left\Vert x\right\Vert \right)
\end{equation*}%
for all $t\geq 0$, $x \in {\mathbb{R}}^d$, and some nondecreasing function $%
\widetilde{c}:[0,\infty )\rightarrow \lbrack 0,\infty )$. Then, for any $T>0$%
, with $\mu (t)=\widetilde{\mu }(t,W(t))$ in the last proposition, the above
computations show the weak existence of a solution to the stochastic
differential equation
\begin{equation*}
X(t)=-\int_{0}^{t}\widetilde{\mu }\left( s,X(s)\right) \mathrm{d}%
s+B(t),\qquad 0\leq t\leq T,
\end{equation*}%
where $B=\{B(t)\}_{t\geq 0}$ denotes a Brownian motion. For an alternative
proof of this statement, using Novikov's condition along with
\textquotedblleft salami tactics,\textquotedblright\ see Proposition~5.3.6
in \citet{KS1}.

The more general assertion of Theorem~\ref{P:Benes} cannot be proven via
this ``salami tactics.'' For example, if $Y$ is a nonnegative pure-jump
supermartingale, then the quadratic variation processes of $Y$ and the
components of $W$ are zero, even if the jump sizes of $Y$ depend, in a
nonanticipative way, on the paths of $W$.

\subsection{Compound Poisson processes}

We continue with an application of Theorem~\ref{T mainmain} to a class of
stochastic differential equations (SDEs) involving jumps. 
Towards this end, for any $\omega \in D_{[0,\infty )}({\mathbb{R}}^{d})$, we
shall write $\Delta \omega (t):=\omega (t)-\omega (t-)$. Further, for any $%
\omega \in D_{[0,\infty )}({\mathbb{R}}^{d})$, define $\omega ^{t}\in
D_{[0,t)}({\mathbb{R}}^{d})$ be equal to the projection of $\omega $ onto $%
D_{[0,t)}\left( {\mathbb{R}}^{d}\right) $; that is, $\omega ^{t}\left(
s\right) =\omega \left( s\right) $ for all $s\in \lbrack 0,t)$. We call a
function $g$ with domain $[0,\infty )\times D_{[0,\infty )}({\mathbb{R}}%
^{d}) $ \emph{predictable} if the function $g\left( \cdot ,\omega \right) $
is measurable for each $\omega \in D_{[0,\infty )}({\mathbb{R}})$, and we
have that $g(t,\omega )=g(t,\varpi )$ for all $t\geq 0$ and all $\omega
,\varpi \in D_{[0,\infty )}({\mathbb{R}}^{d})$ with $\omega ^{t}\equiv
\varpi ^{t}$

Let $F$ denote the distribution of an ${\mathbb{R}}^{d}\setminus \{0\}$%
--valued random variable for some $d\in \mathbb{N}$ and fix $x_{0}\in {%
\mathbb{R}}^{d}$ and a predictable function $g:[0,\infty )\times
D\rightarrow {\mathbb{[}}0,\infty )$. Define $\Psi_g(t,\omega
):=\int_{0}^{t}g(s,\omega )\mathrm{d}s$ for all $t\geq 0$ and $\omega \in
D_{[0,\infty )}({\mathbb{R}}^{d})$. We say that a filtered probability space
$(\Omega ,\mathcal{F},\mathbb{F},P)$ along with an adapted process $X$ with
cadlag paths in ${\mathbb{R}}^{d}$ is a weak solution to the SDE
\begin{align}  \label{E SDE Poisson}
X(t)=x_{0}+\sum_{j=1}^{N_g( t) }Z_{j}^{F},
\end{align}%
if $X(0)=x_{0}$, the jumps $\{\Delta X\mid \Delta X\neq 0\}$ of $X$ are
independent and identically distributed according to $F$, and $L_g\left(
\cdot \right) :=N_g\left( \cdot \right) -\Psi _g(\cdot ,X)$ is a $P$--local
martingale up to the first hitting time of infinity by $N_g$, where $%
N_g\left( \cdot \right) :=\sum_{s\leq \cdot }\mathbf{1}_{\{\Delta X(s)\neq
0\}}$ is the sum of jumps.  Theorem~3.6 in \citet{Jacod_1975} yields the existence and uniqueness of a weak solution to \eqref{E SDE Poisson};
however, such solution might be explosive in the sense that $N_g\left(
t\right) =\infty $ for some $t\in \left( 0,\infty \right) $ with positive
probability. Below, in Lemma~\ref{L nonexplosive}, we will provide
sufficient conditions to ensure a non-explosive solution.

Any non-explosive solution $(\Omega ,\mathcal{F},\mathbb{F},P),X$ of %
\eqref{E SDE Poisson} corresponds to a compound Poisson process with jumps
distributed according to $F$ such that its instantaneous intensity to jump
at time $t$ equals $g(t,X)$; more precisely
\begin{equation*}
\sum_{s\leq \cdot }\mathbf{1}_{\{\Delta X(s)\neq 0\}}=N\left( \Psi _g(\cdot
,X)\right)
\end{equation*}%
for some Poisson process $N=\{N\left( t\right) \}_{t\geq 0}$ with unit rate.

Such a non-explosive solution exists, for example, if $g(t,X)=\mathfrak{g}(t)
$ only depends on time and $\int_{0}^{t}\mathfrak{g}(s)\mathrm{d}s<\infty $
for all $t\geq 0$. The following lemma gives another existence result:

\begin{lemma}
\label{L nonexplosive} Fix $d\in \mathbb{N}$ and $x_{0}\in {\mathbb{R}}^{d}$
and let $F$ denote the distribution of an ${\mathbb{R}}^{d}\setminus \{0\}$%
--valued random variable whose components have finite expected value. Let $%
\mathfrak{g}:{\mathbb{R}}^{d}\rightarrow {\mathbb{R}}$ be measurable, such
that there exists $c>0$ with $\mathfrak{g}(y)\leq c(1+\Vert y\Vert )$ for
all $y\in {\mathbb{R}}^{d}$. Then there exists a non-explosive weak solution
to \eqref{E SDE
Poisson} with $g(t,\omega )=\mathfrak{g}(\omega (t-))$ for all $(t,\omega
)\in \lbrack 0,\infty )\times D$.
\end{lemma}

\begin{proof}
Let $N=\{N\left( t\right) \}_{t\geq 0}$ denote a Poisson process with unit
rate and $Z=\{Z_{j}^{F}\}_{j\geq 1}$ a sequence of independent $F$%
--distributed random variables independent of $N$. First observe that
\begin{equation*}
J\left( t\right) :=x_{0}+\sum_{j=1}^{N\left( t\right) }Z_{j}^{F}
\end{equation*}%
always exists and that
\begin{equation*}
\Gamma (t):=\int_{0}^{t}\frac{1}{\mathfrak{g}(J(s))}\mathrm{d}s\geq \frac{1}{c}%
\int_{0}^{t}\frac{1}{1+\Vert J(s)\Vert }\mathrm{d}s
\end{equation*}%
is strictly increasing (before hitting infinity) and satisfies $\lim_{t \uparrow \infty} \Gamma(t) = \infty$ since there exists $K\left(
\omega \right) \in \left( 0,\infty \right) $ such that $\left\Vert
J(t)\right\Vert \leq K\left( \omega \right) \left( 1+t\right) $ by the law
of large numbers. Thus, $\Gamma $ yields a valid time change. Now, consider the non-explosive process $%
X(t)=J(\Gamma ^{-1}(t))$ and observe that $\dot{\Gamma}^{-1}\left( s\right)
=1/\dot{\Gamma}(\Gamma ^{-1}(s))$, which implies
\begin{equation*}
\Gamma ^{-1}(t)=\int_{0}^{t}\dot{\Gamma}^{-1}\left( s\right) ds=\int_{0}^{t}%
\mathfrak{g}(J(\Gamma ^{-1}(s)))\mathrm{d}s=\int_{0}^{t}\mathfrak{g}(X(s))%
\mathrm{d}s=\psi_g(t,X),
\end{equation*}%
which in turn verifies that $X$ is a non-explosive solution to \eqref{E SDE Poisson}.
\end{proof}

The next theorem provides a sufficient condition that guarantees that the
intensity in Poisson processes can be changed without changing the nullsets
of the underlying probability measure. For example, any compound Poisson
process with a strictly positive intensity can be changed, via an equivalent
change of measure, to a compound Poisson process with unit intensity (set $%
g_{2} \equiv1$ below).

\begin{theorem}
\label{T compound} Fix $d\in \mathbb{N}$ and $x_{0}\in {\mathbb{R}}^{d}$ and
let $F$ denote the distribution of an ${\mathbb{R}}^{d}\setminus \{0\}$%
--valued random variable. Moreover, let $g_{1},g_{2}:[0,\infty )\times
D_{[0,\infty )}({\mathbb{R}}^{d})\rightarrow (0,\infty )$ be strictly
positive, predictable functions and denote the corresponding weak solutions
of \eqref{E SDE Poisson} with $g\equiv g_{1}$ and $g\equiv g_{2}$ by $X_{1}$%
, and $X_{2}$. Assume that $X_{2}$ is non-explosive.

Then, the process $M=\{M\left( t\right) \}_{t\geq 0}$, defined by
\begin{equation*}
M(t):=\exp \left( \int_{0}^{t}(\log g_{2}(s,X_1 )-\log g_{1}(s,X_1 ))\mathrm{%
d}L_{g_{1}}\left( s\right) -\int_{0}^{t}(g_{2}(s,X_1 )-g_{1}(s,X_1 ))\mathrm{%
d}s\right) ,
\end{equation*}%
is a true martingale; furthermore, under $Q$, defined on $\mathcal{F}(t)$ by
$\mathrm{d}Q|_{\mathcal{F}(t)}=M(t)\mathrm{d}P|_{\mathcal{F}(t)}$, the
distribution of $X_{1}$ equals the distribution of $X_{2}$.
\end{theorem}

\begin{proof}
Theorem~VI.2 in \citet{Bremaud_1981} yields that $M$ is a local martingale.
If $M$ is a true martingale, then Theorem~VI.3 in \citet{Bremaud_1981}
yields the assertion on the distribution of $X_{1}$ under the probability
measure $Q$. Define the approximating sequence $\{\tau _{n}\}_{n\in \mathbb{N%
}}$ of stopping times via
\begin{equation*}
\tau _{n}=\inf \{t\geq 0:M\left( t\right) \geq n\text{ or }M\left( t\right)
\leq 1/n\}
\end{equation*}%
and note that we can write those stopping times as functions of the jump process $X_1$; to wit, $\tau_n = \tau_n(X_1)$.
We have included the lower bound to deal with the case in which $X_{1}$ is
explosive; in such case, $M$ will hit zero at the time of the explosion.
Note that such explosion, if it ever occurs, cannot occur at the time of a
jump; thus the local martingale $M_{n}=M^{\tau _{n}}$ is strictly positive.

Next, fix $t>0$. By Theorem~\ref{T mainmain}
it is now sufficient to show that $\{M_{n}(t)\}_{n\in \mathbb{N}}$ is tight
along the sequence of probability measures $\{Q_{n}\}_{n\in \mathbb{N}}$,
defined via $\mathrm{d}Q_{n}=M_{n}(t)\mathrm{d}P$ to obtain the martingale property
of $M$. For $i=1,2$, we shall see that
\begin{equation*}
\left\{ \int_{0}^{\tau _{n}\wedge t}\left\vert \log g_{i}(s,\omega
)\right\vert \mathrm{d}L_{g_{1}}\left( s\right) \right\} _{n\in \mathbb{N}}%
\text{ and }\left\{ \int_{0}^{\tau _{n}\wedge t}g_{i}(s,\omega )\mathrm{d}%
s\right\} _{n\in \mathbb{N}}
\end{equation*}%
are tight along $\{Q_{n}\}_{n\in \mathbb{N}}$, which then proves the
statement. Towards this end, Theorem~VI.3 in \citet{Bremaud_1981} again
yields that under $Q_{n}$ the process $X_{1}( \cdot \wedge \tau _{n}) $
solves the martingale problem induced by \eqref{E SDE
Poisson} with $g( s,\omega ) =g_{2}( s,\omega )\1_{\{ \tau _{n}( \omega ) >s\}} $. On the other hand, it is
immediate that $X_{2}\left( \cdot \wedge \tau _{n}\right) $ also satisfies
\eqref{E SDE
Poisson} with $g( s,\omega ) =g_{2}( s,\omega )
\1_{\{ \tau _{n}( \omega ) >s\}} $. By the uniqueness of
solutions implied by Theorem~3.6 in \citet{Jacod_1975} we have that up to
the stopping time $\tau _{n}$, the $Q_{n}$-dynamics of $X_{1}$ coincide with
the dynamics of $X_{2}$. Thus, it is sufficient to observe that
\begin{align*}
Q_{n}\left( \int_{0}^{\tau _{n}\wedge t}\left\vert \log
g_{i}(s,X_{1})\right\vert \mathrm{d}L_{g_{1}}\left( s\right) >\kappa \right)
& =P\left( \int_{0}^{\tau _{n}\wedge t}\left\vert \log
g_{i}(s,X_{2})\right\vert \mathrm{d}L_{g_{2}}\left( s\right) >\kappa \right)
\\
& \leq P\left( \int_{0}^{t}\left\vert \log g_{i}(s,X_{2})\right\vert \mathrm{%
d}L_{g_{2}}\left( s\right) >\kappa \right)
\end{align*}%
for all $\kappa >0$, where the right-hand side does not depend on $n$ and
tends to zero as $\kappa $ increases (because $X_{2}$ is assumed to be
non-explosive). The same observations hold for the other terms of the local martingale $M$.
\end{proof}

\subsection{Counting processes}

In this application of Theorem~\ref{T mainmain}, we generalize a result by %
\citet{Giesecke_2013} concerning the martingale property of a local
martingale involving a counting process.

\begin{theorem}
\label{T counting} Let $L=\{L(t)\}_{t\geq 0}$ denote a non-explosive
counting process with compensator $A=\{A(t)\}_{t\geq 0}$ and assume that $A$
is continuous, that is, the jumps of $L$ are totally inaccessible. Fix a
measurable, deterministic function $u:[0,\infty )\rightarrow \lbrack -c,c]$
for some $c>0$. Then the process $M=\{M(t)\}_{t\geq 0}$, given by
\begin{equation*}
M(t):=\exp \left( -\int_{0}^{t}u(s)\mathrm{d}L(s)-\int_{0}^{t}\left( \exp
(-u(s))-1\right) \mathrm{d}A(s)\right) 
\end{equation*}%
for all $t\geq 0$, is a martingale.
\end{theorem}

Before we provide the proof of this result we note that Theorem~\ref{T
counting} generalizes Proposition~3.1 in \citet{Giesecke_2013} in two ways.
First, it does not assume that the function $u$ is constant. Second, no
integrability assumption on $A$ is made, such as $E[\exp (A_{t})]<\infty $
for some $t>0$. However, for sake of simplicity, we only considered the
one-dimensional setup with unit jumps.

\begin{proof}[Proof of Theorem~\ref{T counting}]
First, observe that there exists a counting process $\widehat{L}$, possible
on an extension of the probability space, with compensator $\widehat{c}A$,
where $\widehat{c}$ is the smallest integer greater than or equal to $\exp
(c)$. For example, the process $\widehat{L}$ can by constructed by adding $%
\widehat{c}$ independent versions of $L$, exploiting the fact that the jumps
of $L$ are totally inaccessible. A standard thinning argument implies that
there also exists a counting process $L^{u}$ with compensator $%
A^{u}:=\int_{0}^{t}\exp (-u(s))\mathrm{d}A(s)$. Moreover, by %
\citet{Jacod_1975} and by using the minimal filtration, if two counting
processes $L^{u}$ and $\widehat{L}^{u}$ have the same compensator then they
follow the same probability law.

Simple computations yield that $M$ is a local martingale. Let $\{\tau _{n}\}$
denote a localization sequence, set $M_{n}=M^{\tau _{n}}$, fix $t>0$, and
define the probability measures $Q_{n}$ by $\mathrm{d}Q_{n}=M_{n}(t)\mathrm{d%
}P$. It is sufficient to prove that $\{M_{n}(t)\}_{n\in \mathbb{N}}$ is
tight along the sequence $\{Q_{n}\}_{n\in \mathbb{N}}$. First, observe that
\begin{equation*}
Q_{n}\left( \exp \left( -\int_{0}^{\tau _{n}}u(s)\mathrm{d}L(s)\right) \geq
\kappa \right) \leq Q_{n}\left( \exp \left( cL(\tau _{n})\right) \geq \kappa
\right) \leq P\left( \exp (c\widehat{L}(t))>\kappa \right)
\end{equation*}%
for all $n\in \mathbb{N}$ and $\kappa >0$. For later use, observe that we
also have the tightness of $\{L^{\tau _{n}}(t)\}_{n\in \mathbb{N}}$ along
the sequence $\{Q_{n}\}_{n\in \mathbb{N}}$. Tightness of $\{M_{n}(t)\}_{n\in
\mathbb{N}}$ now follows as soon as we have shown the tightness of $%
\{A^{\tau _{n}}(t)\}_{n\in \mathbb{N}}$, that is, $\sup_{n\in \mathbb{N}%
}Q_{n}(A^{\tau _{n}}(t)\geq \kappa )\rightarrow 0$ as $\kappa \uparrow
\infty $. However, we can write
\begin{equation*}
A^{\tau _{n}}(t)=\int_{0}^{t\wedge \tau _{n}}\exp (u(s\wedge \tau _{n}))%
\mathrm{d}A^{u}(s)\leq \exp (c)A^{u,\tau _{n}}(t)
\end{equation*}%
and thus, with $\kappa _{c}:=\exp (-c)\kappa $ and $N=A^{u,\tau
_{n}}-L^{\tau _{n}}$,
\begin{align*}
Q_{n}\left( A^{\tau _{n}}(t)\geq \kappa \right) & \leq Q_{n}\left( A^{u,\tau
_{n}}(t)\geq \kappa _{c}\right) =Q_{n}\left( N(t)+L^{\tau _{n}}(t)\geq
\kappa _{c}\right)  \\
& \leq Q_{n}\left( L^{\tau _{n}}(t)\geq \lbrack \sqrt{\kappa _{c}}]\right)
+Q_{n}\left( N(t\wedge \rho )+[\sqrt{\kappa _{c}}]+1\geq \kappa _{c}+1\right) ,
\end{align*}%
where $\rho $ is the first hitting time of $[\sqrt{\kappa _{c}}]$ by $L$.
The tightness of $\{L^{\tau _{n}}(t)\}_{n\in \mathbb{N}}$ and Markov's
inequality applied to the nonnegative $Q_{n}$--supermartingale $N^{\rho }+[%
\sqrt{\kappa _{c}}]+1$ then yield the tightness of $\{A^{\tau
_{n}}(t)\}_{n\in \mathbb{N}}$ and Theorem~\ref{T mainmain} yields the
statement.
\end{proof}

\subsection{Ornstein-Uhlenbeck process given rare first passage time events}

In this application, we are given an Ornstein-Uhlenbeck process $X$, started at $X(0)=1$
and mean-reverting to the origin. We are interested in finding a representation for conditional expectations that can be used to design simulation estimators involving the rare event that $X$ hits a large level $N\in \mathbb{N}$ before hitting $0$. Such questions arise in studying
overflow probabilities within operational cycles engineering systems, such as queueing networks. 
Ornstein-Uhlenbeck processes arise in such a setup as an
approximative description of a system with infinitely many servers in heavy traffic, see Chapter~6 of \citet{Robert2003stochastic}.

We achieve such a representation by relating the Ornstein-Uhlenbeck process to the
time-reversal of a three-dimensional Bessel process. Although the
probability of the conditioning event that $X$ hits $N$ before $0$ decreases
\emph{exponentially} in the threshold parameter $N$, as we note in Remark 1 below, the representation provided here
can be used to design estimators that run in \emph{linear} time as a function of $N$. To obtain this representation, we
use a result of \citet{Blanchet2013} for irreducible and positive recurrent
discrete-time Markov chains. We then approximate $X$ by a sequence of such Markov chains,
apply the result of \citet{Blanchet2013}, and then use the second part of
Theorem~\ref{T mainmain} to conclude.

We first recall Proposition~1 in \citet{Blanchet2013}:

\begin{proposition}
\label{Conv} Let $X=\{X_{k}\}_{k\in \mathbb{N}_{0}}$ denote an irreducible
and positive recurrent discrete time Markov chain taking values in some
countable state space $\mathcal{S}$ and having stationary distribution $\pi $%
. Let $X^{\prime }=\{X_{k}^{\prime }\}_{k\in \mathbb{N}_{0}}$ denote the
time-reversal of $X$. For each $x\in \mathcal{S}$, let $P_{x}$ denote the
probability measure in the path space associated with $X$, conditioned on
the event $\{X_{0}=x\}$. Let $P_{\pi }^{\prime }$ denote the probability
measure associated with $X^{\prime }$ when started in the stationary
distribution. Fix a function $V:\mathcal{S}\longrightarrow \mathbb{R}$, $%
N\in {\mathbb{R}}$, and $x,b\in \mathcal{S}$ with $V(x)<N$ and $V\left(
b\right) <N$. For any $y\in \mathcal{S}$ define $T_{y}=\inf \{k\in \mathbb{N}%
:X_{k}=y\}$ and, similarly, $T_{y}^{\prime }$. Write $T_{\ast }=\inf \{k\in
\mathbb{N}:V(X_{k})\geq N\}$ and define $T_{\ast }^{\prime }$ similarly.

Then
\begin{equation*}
P_{x}\left( \left. \left( X_{0},\ldots ,X_{T_{\ast }}\right) \in \cdot
\right\vert \{T_{\ast }<T_{b}\}\right) =P_{\pi }^{\prime }\left( \left.
\left( X_{\xi ^{\prime }(x)}^{\prime },\ldots ,X_{0}^{\prime }\right) \in
\cdot \right\vert \{V\left( X_{0}^{\prime }\right) \geq N\}\cap
\{T_{x}^{\prime }\leq T_{b}^{\prime }<T_*'\}\right) ,
\end{equation*}%
where $\xi ^{\prime }\left( x\right) =\max \{0\leq k\leq T_{b}^{\prime
}:X_{k}^{\prime }=x\}$.
\end{proposition}

Proposition~\ref{Conv} states that we can sample $\{X_{k}\}_{0\leq k\leq
T_{\ast }}$ conditioned on the event $\{X_{0}=x\}\cap \{T_{\ast }<T_{b}\}$
by first sampling $X_{0}^{\prime }$ from $\pi $ conditioned on the event $%
\{V\left( X_{0}^{\prime }\right) \geq N\}$, then sampling $\{X_{k}^{\prime
}:1\leq k\leq T_{b}^{\prime }\}$ conditioned on the event $\{T_{x}^{\prime
}\leq T_{b}^{\prime }<T^{\ast \prime }\}$, thereby obtaining
\begin{equation*}
\{X_{k}^{\prime }:0\leq k\leq \xi ^{\prime }\left( x\right) \},
\end{equation*}%
and finally letting $X_{k}=X_{\xi ^{\prime }\left( x\right) -k}^{\prime }$
for $k\in \{0,...,\xi ^{\prime }\left( x\right) \}$.

We do not provide a proof of Proposition~\ref{Conv} here, but instead refer
to \citet{Blanchet2013}. However, to provide some intuition, we give some
computations here, which indicate the validity of the result. Towards this
end, let $\{K(x,y)\}_{x,y\in \mathcal{S}}$ denote the transition matrix of $X
$ and $\{K^{\prime }(x,y)\}_{x,y\in \mathcal{S}}$ the one of $X^{\prime }$.
Recall that
\begin{equation*}
K^{\prime }\left( y,x\right) =\pi \left( x\right) K\left( x,y\right) /\pi
\left( y\right) .
\end{equation*}%
Then, note that
\begin{align*}
\pi \left( b\right) P_{b}\left( T_{\ast }<T_{b}\right) & =\sum_{k\in \mathbb{%
N}}\sum_{x_{1},\ldots ,x_{k}\in \mathcal{S}}\pi \left( b\right) K\left(
b,x_{1}\right) K\left( x_{1},x_{2}\right) \cdots K\left(
x_{k-1},x_{k}\right)  \\
& \quad \times \mathbf{1}_{x_{1}\neq b,V\left( x_{1}\right) <N,\ldots
,x_{k-1}\neq b,V\left( x_{k-1}\right) <N,V\left( x_{k}\right) \geq N} \\
& =\sum_{k\in \mathbb{N}}\sum_{x_{1},\ldots ,x_{k}\in \mathcal{S}}K^{\prime
}\left( x_{1},b\right) \pi \left( x_{1}\right) K\left( x_{1},x_{2}\right)
\cdots K\left( x_{k-1},x_{k}\right)  \\
& \quad \times \mathbf{1}_{x_{1}\neq b,V\left( x_{1}\right) <N,\ldots
,x_{k-1}\neq b,V\left( x_{k-1}\right) <N,V\left( x_{k}\right) \geq N} \\
& \cdots  \\
& =\sum_{k\in \mathbb{N}}\sum_{x_{1},\ldots ,x_{k}\in \mathcal{S}}\pi \left(
x_{k}\right) K^{\prime }\left( x_{k},x_{k-1}\right) \cdots K^{\prime }\left(
x_{2},x_{1}\right) K^{\prime }\left( x_{1},b\right)  \\
& \quad \times \mathbf{1}_{V\left( x_{k}\right) \geq N,x_{k-1}\neq b,V\left(
x_{n-1}\right) <N,\ldots ,x_{1}\neq b,V\left( x_{1}\right) <N} \\
& =\sum_{k\in \mathbb{N}}\sum_{x_{0}^{\prime },\ldots ,x_{k-1}^{\prime }\in
\mathcal{S}}\pi \left( x_{0}^{\prime }\right) K^{\prime }\left(
x_{0}^{\prime },x_{1}^{\prime }\right) \cdots K^{\prime }\left(
x_{k-2}^{\prime },x_{n-1}^{\prime }\right) K^{\prime }\left( x_{k-1}^{\prime
},b\right)  \\
& \quad \times \mathbf{1}_{V\left( x_{0}^{\prime }\right) \geq
N,x_{1}^{\prime }\neq b,V\left( x_{1}^{\prime }\right) <N,\ldots
,x_{k-1}^{\prime }\neq b,V\left( x_{k-1}^{\prime }\right) <N} \\
& =E_{\pi }^{\prime }[P_{X_{0}^{\prime }}^{\prime }(T_{b}^{\prime }<T_{\ast
}^{\prime })\text{ }|\text{ }V\left( X_{0}^{\prime }\right) \geq N]P_{\pi
}^{\prime }\left( V\left( X_{0}^{\prime }\right) \geq N\right)
\end{align*}%
with the obvious notation for $E_{\pi }^{\prime }$. The previous identities
provide a representation for $P_{b}(T_{\ast }<T_{b})$ in terms of $P_{\pi
}^{\prime }\left( V\left( X_{0}^{\prime }\right) \geq N\right) $ and the
expectation involving the probability $P_{X_{0}^{\prime }}^{\prime
}(T_{b}^{\prime }<T_{\ast }^{\prime })$, where $X_{0}^{\prime }$ satisfies $%
V\left( X_{0}^{\prime }\right) \geq N$. It is shown in \citet{Blanchet2013}
that the contribution of
\begin{equation*}
E_{\pi }^{\prime }[P_{X_{0}^{\prime }}(T_{b}^{\prime }<T_{\ast }^{\prime
})|V\left( X_{0}^{\prime }\right) \geq N]
\end{equation*}%
remains bounded away from zero as $N$ increases to $\infty $ for a
significant class of processes of interest. Therefore, computing and
sampling rare event probabilities for sample path events of the form $%
\{T_{\ast }<T_{b}\}$ can be reduced to computing rare event probabilities
for the random variable $X_{0}^{\prime }$ following the stationary
distribution, via the event $\{V\left( X_{0}^{\prime }\right) \geq N\}$.

Note that all these computations are tailored to discrete-time processes,
and cannot very easily be extended to continuous processes. Our goal in this
section is to use Theorem~\ref{T mainmain} in order to obtain a suitable
analogue of Proposition~\ref{Conv} for continuous processes. We will not
provide full details of an extension in general, but will focusing on
proving a tractable representation for an Ornstein-Uhlenbeck process
conditioned on reaching a high level before returning to the origin.
Tractable means that the representation should be directly applicable for
the purposes of sampling.

\begin{theorem}
\label{T OU} Fix $N \in \mathbb{N}$ with $N>1$, a filtered probability space
$(\Omega, \mathcal{F}, \{\mathcal{F}_t\}_{t \geq 0}, P)$ with expectation
operator $E$, supporting four independent Brownian motions $B_i =
\{B_{i}(t)\}_{t\geq 0}$ for $i = 0, \ldots, 3$. Let $X=\{ X(t)\}_{t \geq 0}$ denote an Ornstein-Uhlenbeck process of the form
\begin{align*}
X( t) =1-\int_{0}^{t}X\left( s\right) \mathrm{d}s+B_0(t) ,
\end{align*}
and $X^{\prime }=\{ X^{\prime }(t)\}_{t \geq 0}$ be given by
\begin{align*} 
X^{\prime }( t) &=N-\left( B_{1}^{2}\left( t\right) +B_{2}^{2}\left(
t\right) +B_{3}^{2}\left( t\right) \right) ^{1/2}.
\end{align*}
For all $x \in {\mathbb{R}}$, let $T_x=\inf \{t\geq 0:X( t) =x\}$ and define
$T_x^{\prime }$ similarly. Define a random variable $M^{\prime }(\infty)$ by
\begin{align}  \label{eq:M'}
M^{\prime }\left( \infty \right) &=\exp \left( \frac{1}{2}\left(N^2 +
T_{0}^{\prime }-\int_{0}^{T_{0}^{\prime }}X^{\prime }\left( s\right) ^{2}%
\mathrm{d}s\right) \right)
\end{align}

Then the random variable $M'(\infty)$ has finite expectation under $P$ and
\begin{equation*}
E\left[ f\left( X\left( s\right) :0\leq s\leq T_{N}\right) |\{T_{N}<T_{0}\}%
\right] ={E}\left[ f\left( X^{\prime }\left( \xi ^{\prime }\left( 1\right)
-s\right) :0\leq s\leq \xi ^{\prime }\left( 1\right) \right) \frac{M^{\prime
}\left( \infty \right) }{E[M^{\prime }\left( \infty \right) ]}\right] ,
\end{equation*}%
where $\xi ^{\prime }\left( 1\right) =\max \{0\leq t\leq T_{0}:X^{\prime
}(t)=1\}$, for all continuous and
bounded functions  $f:D_{[0,\infty )}\rightarrow \mathbb{R}$.
\end{theorem}

\begin{remark}
The previous result can be used to efficiently estimate conditional
expectations involving Ornstein-Uhlenbeck processes, conditioned on $\{T_{N}<T_{0}\}$ when $N$ is large, in a way that is
analogous to the methods described in \citet{Blanchet2013}. This then  leads to
algorithms that have linear running time uniformly as $N\uparrow \infty $.
This approach will be studied
in future work.\qed
\end{remark}

\begin{remark}
It is well known that $X^{\prime }$, the modified three-dimensional Bessel
process in Theorem~\ref{T OU}, satisfies the stochastic differential
equation
\begin{align}  \label{eq:Bessel}
X^{\prime }( \cdot) &=N-\int_0^\cdot \frac{1}{N-X^{\prime }( t) } \mathrm{d}
t+ B(\cdot)
\end{align}%
for some Brownian motion $B(\cdot) = \{B(t)\}_{t \geq 0}$.

Also, note that
\begin{align*}
M^{\prime }\left( \infty \right) &=\exp \left( -\int_0^{T_0^{\prime }}
X^{\prime }(s) \mathrm{d} X^{\prime }(s) - \frac{1}{2}\int_{0}^{T_{0}^{%
\prime }}X^{\prime }\left( s\right) ^{2}\mathrm{d}s\right) \\
&= \exp \left( -\int_0^{T_0^{\prime }} X^{\prime }(s) \mathrm{d} B(s) -
\frac{1}{2}\int_{0}^{T_{0}^{\prime }}X^{\prime }\left( s\right) ^{2}\mathrm{d%
}s\right) \exp\left(\int_0^{T_0^{\prime }} \frac{X^{\prime }(s)}{N-
X^{\prime }(s)} \mathrm{d} s\right),
\end{align*}
where the first equality follows from an application of It\^o's lemma and
the last equality from \eqref{eq:Bessel}. \qed
\end{remark}

\begin{proof}[Proof of Theorem~\ref{T OU}]
We will consider a suitably defined class of discrete processes that
approximate $X\left( \cdot \right) $ and then apply Proposition~\ref{Conv}.
More precisely, we construct a sequence of stochastic processes $%
\{X_{n}\}_{n \in \mathbb{N}}$ with $X_{n}=\{X_{n}(t)\}_{t\geq 0}$, each
taking values in the state space $\mathcal{S}_{n}=\{k\Delta
_{n}^{1/2}\}_{k\in \mathbb{N}_{0}}$, where $\Delta _{n}:=2^{-2n}$. For each $%
n \in \mathbb{N}$, we let $X_{n}$ evolve as a pure jump process, jumping
only at times $\{k\Delta _{n}\}_{k\in \mathbb{N}}$ with
\begin{equation*}
\Delta X_{n}\left( k\Delta _{n}\right) =\Delta _{n}^{1/2}J_{n}\left(
k,X_{n}\left( (k-1)\Delta _{n}\right) \right) ,
\end{equation*}%
where $\{J_{n}(k,\cdot )\}_{k\in \mathbb{N}}$ satisfies
\begin{equation*}
J_{m}(k,y)=\left( 2\mathbf{1}_{\{U(k)\leq 2^{-1}(1-\Delta _{n}^{1/2}(y\wedge
n)))\}}-1\right) \mathbf{1}_{y>0}+\mathbf{1}_{y=0} \in \{-1,1\},
\end{equation*}%
for all $y \geq 0$ and $\{U\left( k\right) \}_{k\in \mathbb{N}}$ is a
sequence of uniformly distributed in $(0,1)$ i.i.d.~random variables. In
simple words, as long as $X_{n}\left( k\Delta _{n}\right) >0$, the next
increment is $+1$ with probability $(1-q_n^k)/2\in (0,1)$ and $-1$ with
probability $(1+q_n^k)/2\in (0,1)$, where $q_n^k = \Delta
_{n}^{1/2}(X_{n}(k\Delta _{n})\wedge n)$. For each $n \in \mathbb{N}$, we
denote the induced probability distribution of $X_n$ on the canonical path space and the corresponding
expectation operator by $P_n$ and $E_n$. Note that $E_n[J_{n}(k,y)]=-\Delta
_{n}^{1/2}(y\wedge n)$ for all $k \in \mathbb{N}$ and $y \geq 0$. Moreover,
we define $\xi ^{n}\left( 1\right) =\max \{0 \leq t \leq T_0^n:X_{n}\left(
t\right) =1\}$, where $T_{s}^n=\inf \{t\geq 0:X_n\left( t\right) =s\}$ for
all $s\in \mathcal{S}_{n}$ and $n \in \mathbb{N}$.

We now fix $n \in \N$.
Let us introduce the probability measure $\widehat{P}_{n}$ under which $%
J_{n}(\cdot ,\cdot )$ increases by $1$ or $-1$ with probability $1/2$ until
time $T_{0}^{n}$. Let us also introduce the
random variable
\begin{equation*}
M_{n}^{\prime }\left( \infty\right) =\prod_{k=1}^{ T_{0}^{n}/\Delta
_{n}}\left( 1-q_{n}^{k-1}J_{n}\left( k,X_{n}\left( \left( k-1\right) \Delta
_{n}\right) \right) \right).
\end{equation*}%
It is clear that
\begin{equation*}
\widehat{E}_{n}\left[ 1-q_{n}^{k-1}J_{n}\left( k,X_{n}\left( \left(
k-1\right) \Delta _{n}\right) \right) |X_{n}\left( \left( k-1\right) \Delta
_{n}\right) \right] =1
\end{equation*}%
for all $k\in \mathbb{N}$ and that $P(T_{0}^{n}<\infty )=1$, which yields
\begin{equation*}
\widehat{E}_{n}\left[ M_{n}^{\prime }\left( \infty \right) \right]
=\sum_{k=1}^{\infty }\widehat{E}_n\left[ M_{n}^{\prime }\left( \infty
\right) \mathbf{1}_{\{T_{0}^{n}=k\}}\right] =\sum_{k=1}^{\infty }{P}%
_n\left( T_{0}^{n}=k\right) =1.
\end{equation*}%
This implies that $P_n\ll \widehat{P}_n$ and $%
M_{n}^{\prime }\left( \infty \right) ={\mathrm{d}P_n}/{\mathrm{d}\widehat{P%
}_n}$.

Let us write
\begin{equation*}
\widehat{P}^{h}_n(\cdot )=\widehat{P}_n(\cdot | \{X_{n}\left( 0\right) =N\}
\cap \{T_{0}^{n}<T_{N}^{n}\}).
\end{equation*}
In order to describe the conditional dynamics of $X$ under $\widehat{P}^{h}_n
$, we apply Doob's $h$-transform and define the function $h^n:[0,N]
\rightarrow (0,1)$ by
\begin{align*}
h^n\left( x\right) =\widehat{P}_n(T_{0}^{n}<T_N^{n}|\{X_{n}( 0)
=x\})=\frac{N-x}{N}
\end{align*}
for all $x \in [0,N]$. Observe that under $\widehat{P}%
^{h}_n$ the random variable $J_n(k,y)$ is $1$ with probability
\begin{align*}
\frac{h^n(y+\Delta _{n}^{1/2})}{2h(y )}=\frac{1}{2}\left( 1-\frac{\Delta
_{n}^{1/2}}{N-y }\right)
\end{align*}
and $-1$ with probability
\begin{align*}
\frac{h^n(y-\Delta _{n}^{1/2})}{2h(y )}=\frac{1}{2}\left( 1+\frac{\Delta
_{n}^{1/2}}{N-y }\right),
\end{align*}
conditional on the event $\{T_0^n \wedge T_N^n > k \Delta_n^{1/2}\}$, for all $k \in \N$.

For any random variable $H$ depending only on $\{X_{n}( k\Delta_n)
:0\leq k\leq T_{0}^{n}/\Delta \}$ we have
\begin{align}
E_{n}\left[ H|\{X_{n}\left( 0\right) =N\}\cap \{T_{0}^{n}<T_{N}^{n}\}\right]
& =\frac{E_{n}\left[ \left. H\mathbf{1}_{\{T_{0}^{n}<T_{N}^{n}\}}\right\vert
\{X_{n}\left( 0\right) =N\}\right] }{P_{n}\left(
T_{0}^{n}<T_{N}^{n}|\{X_{n}\left( 0\right) =N\}\right) } \nonumber \\
& =\frac{\widehat{E}_{n}\left[ \left. H\mathbf{1}_{\{T_{0}^{n}<T_{N}^{n}%
\}}M_{n}^{\prime }\left( \infty \right) \right\vert \{X_{n}\left( 0\right)
=N\}\right] }{P_{n}\left( T_{0}^{n}<T_{N}^{n}|\{X_{n}\left( 0\right)
=N\}\right) }\nonumber \\
& =\widehat{E}_{n}\left[ \left. HM_{n}^{\prime }\left( \infty \right)
\right\vert \{X_{n}\left( 0\right) =N\}\cap \{T_{0}^{n}<T_{N}^{m}\}\right]
\frac{\widehat{P}_{n}(T_{0}^{n}<T_{N}^{n}|\{X_{n}\left( 0\right) =N\})}{%
P_n\left( T_{0}^{n}<T_{N}^{n}|\{X_{n}\left( 0\right) =N\}\right) } \nonumber\\
& =\widehat{E}^{h}_{n}\left[ HM_{n}^{\prime }\left( \infty \right) \right]
\frac{1}{\widehat{E}^{h}_{n}[M_{n}^{\prime }\left( \infty \right)]},  \label{eq:6}
\end{align}%
where the last equality follows from the definition of $\widehat{P}^{h}_{n}$
and using $H=1$.

The process $X_{n}$, being a birth-death process, is
time reversible. Thus,
Proposition~\ref{Conv} yields that for each continuous bounded function $f: D_{[0,\infty )} \rightarrow R$,%
\begin{align*}
& E_{n}\left[ f(X_{n}(s):0\leq s\leq T_{N})|\{X_{n}\left( 0\right) =1\}\cap
\{T_{N}^{n}<T_{0}^{n}\}\right]  \\
& \quad =E_{n}\left[ f\left( X_{n}\left( \xi _{n}(1)-s\right) :0\leq s\leq
\xi _{n}(1)\right) |\{X_{n}\left( 0\right) =N\}\cap \{T_{0}^{n}<T_{N}^{n}\}%
\right]  \\
& \quad =\widehat{E}^{h}_{n}\left[ f\left( X_{n}\left( \xi _{n}(1)-s\right)
:0\leq s\leq \xi _{n}(1)\right) M_{n}\left( \infty \right) \right] ,
\end{align*}%
where
\begin{align} \label{eq: Mn}
	M_{n}\left( \infty \right) = \frac{M_{n}^{\prime }\left( \infty \right)}{
\widehat{E}^{h}_{n}[M_{n}^{\prime }\left( \infty \right) ]}.
\end{align}

Next, the facts that $1-x=\exp \left( -x-x^{2}/2+O\left( x^{3}\right)
\right) $ as $x\downarrow 0$ and $J_{n}^{2}=1$ imply that
\begin{equation*}
M_{n}^{\prime }\left( \infty \right) =\prod_{k=1}^{T_{0}^{n}/\Delta
_{n}}\exp \left( -q_{n}^{k-1}J_{n}\left( k,X_{n}\left( \left(
k-1\right) \Delta _{n}\right) \right) -\frac{(q_{n}^{k-1})^{2}}{2}%
+O(\Delta _{n}^{3/2})\right)
\end{equation*}%
as $n\uparrow \infty $, where the term $O(\Delta _{n}^{3/2})$ is actually
uniform in $X_{n}\left( ( k-1) \Delta _{n}\right) $ for all $%
k\leq T_{0}^{n}\wedge T_{N}^{n}$.

Note that, if $n>N$, on the event $\{X_{m}\left( 0\right) =N\}\cap
\{T_{0}^{n}<T_{N}^{n}\}$,
\begin{align*}
\frac{1}{2}N^{2}& =\frac{1}{2}\left( X_{n}\left( T_{N }^{n}\right)
-X_{n}\left( 0\right) \right) ^{2} \\
& =\frac{1}{2}\left( \sum_{k=1}^{T_{0 }^{n}/\Delta _{n}}\Delta
_{n}^{1/2}J_{n}\left( k,X_{n}\left( \left( k-1\right) \Delta _{n}\right)
\right) \right) ^{2} \\
& =\frac{1}{2}\Delta _{n}\sum_{k=1}^{T_{0 }^{n}/\Delta
_{n}}J_{n}^{2}\left( k,X_{n}\left( \left( k-1\right) \Delta _{n}\right)
\right)  \\
& \qquad +\sum_{k=1}^{T_{0 }^{n}/\Delta _{n}}\sum_{j=1}^{k-1}\Delta
_{n}^{1/2}J_{n}\left( k,X_{n}\left( \left( k-1\right) \Delta _{n}\right)
\right) \Delta _{n}^{1/2}J_{n}\left( j,X_{n}\left( \left( j-1\right) \Delta
_{n}\right) \right)  \\
& =\frac{1}{2}T_{0 }^{n}+N^{2}+\Delta _{n}^{1/2}\sum_{k=1}^{T_{0
}^{n}/\Delta _{n}}J_{n}\left( k,X_{n}\left( \left( k-1\right) \Delta
_{n}\right) \right) X_{n}\left( \left( k-1\right) \Delta _{m}\right) .
\end{align*}%
Therefore, if $n>N$, on the event $\{X_{n}( 0) =N\}\cap
\{T_{0}^{n}<T_{N}^{n}\}$,
\begin{equation}
M_{n}^{\prime }\left( \infty \right) =\exp \left( \frac{N^{2}}{2}+\frac{%
T_{0}^{n}}{2}-\sum_{k=1}^{T_{0}^{n}/\Delta _{n}}\frac{(q_{n}^{k-1})^{2}}{2}%
+T_{0}^{n}O(\Delta _{n}^{1/2})\right) .  \label{Eq_Mprime}
\end{equation}

It is not difficult to verify using the method of weak convergence of
generators in \citet{EK_Markovian} that%
\begin{align} \label{eq:8}
\left(P_n(\cdot |\{X_n(0) = 1\}),X_{n}\left( \cdot \wedge T_{N}^{n}\wedge T_{0}^{n}\right) \right)\overset{%
\mathfrak{w}}{\Longrightarrow }(P,X\left( \cdot \wedge T_{N}\wedge
T_{0}\right) )  \text{ } (n \uparrow \infty)
\end{align}%
on $D_{[0,\infty)}$.
Proposition~5.33 in \citet{Pitman_1975} implies that
\begin{equation*}
\left(\widehat{P}^{h}_{n},X_{n}\left( \cdot \wedge T_{0}^{n}\right) \right)\overset{%
\mathfrak{w}}{\Longrightarrow }(P,X^{\prime }\left( \cdot \wedge
T_{0}\right) ) \text{ } (n \uparrow \infty)
\end{equation*}%
on $D_{[0,\infty)}$.
The continuous mapping
principle, applied with a standard extension to handle the stopping times $\{T_0^n\}_{n \in \N}$, yields the weak convergence result
\begin{equation*}
\left(\widehat{P}^{h}_{n},M_{n}^{\prime }\left( \infty \right) \right)\overset{%
\mathfrak{w}}{\Longrightarrow }(P,M^{\prime }( \infty ) ) \text{ } (n \uparrow \infty),
\end{equation*}%
where $M'(\infty)$ is defined in \eqref{eq:M'}.

Observe that there exists also a subsequence $\{n_m\}_{m \in \N}$ such that $C = \lim_{m \uparrow \infty} \widehat E^h_{n_m}[M'_{n_m} (\infty)]$ exists in $[0, \infty]$. Then, an application of Fatou's lemma, in conjunction with a Skorokhod embedding argument also yields that $C>0$.  Thanks to the continuous mapping principle, in order to conclude the proof of the statement it is now sufficient to show that  $\{M_{n}(\infty)\}_{n \in \N}$, given in \eqref{eq: Mn}, is tight under the sequence $\{Q_n\}_{n \in \N}$  of probability measures, defined by  $\dd Q_n  = M_n(\infty) \dd \widehat P^h_n$. By Fatou's lemma and by  \eqref{Eq_Mprime}, it is sufficient to show the tightness of $\{T_0^{n}\}_{n \in \N}$ under $\{Q_{n}\}_{n\in \N}$.
For each $n\in \N$, we have
\begin{align}
	Q_n(T_0^n > \kappa) &= \widehat E^h_n\left[\1_{\{\kappa < T_0^n\}} M_n(\infty)\right]
		= \frac{\widehat E_n\left[\left. \1_{\{\kappa < T^n_0 < T_N^n\}} M_n(\infty) \right|\{X_n(0) = N\}\right]}{\widehat{P}_n(T^n_0 < T_N^n| \{X_n(0) = N\})} \nonumber\\
		&= P_n\left( T_{0}^{n}>\kappa |  \{X_{n}(0)=N\} \cap \{T_{0}^{n}<T_{N}^{n}\} \right)     \label{eq:9}\\
&\leq \frac{P_n\left( T_{0}^{n}>\kappa |\{X_{n}\left( 0\right) =1 \}\right)
P_n\left( T_{1 }^{n}<T_{N}^{n}  | \{X_{n}(0)=N\}\right) }{%
P_n\left( T_{0}^{n}<T_{N}^{n} |\{
X_{n}(0)=N\}\right) }   \nonumber\\
&= \frac{P_n\left( T_{0}^{n}>\kappa | \{X_{n}\left( 0\right) =1 \} \right)
}{P_n\left( T_{0}^{n}<T_{N}^{n}  |  \{X_{n}(0)=1 \}\right) },  \label{eq:10}
\end{align}
where the equality \eqref{eq:9} comes from \eqref{eq:6}.

Recall \eqref{eq:8}; therefore%
\begin{equation*}
\liminf_{n \uparrow \infty }P_n\left( T_{0}^{n}<T_{N}^{n}  |  \{X_{n}(0)=1 \} \right) =P\left( T_{0}<T_{N} \right) >0,
\end{equation*}%
and $\left\{ T_{0}^{n}\right\} _{n\in \N}$ is tight under $\{P_{n}\}_{n \in \N}$; thus, $\kappa
>0$ can be chosen so that the right hand side of \eqref{eq:10} can be made
as small as desired as $n\uparrow \infty $.
This concludes the proof.
\end{proof}

\begin{remark}
We end our discussion by noting that the previous result illustrates the
convenience of Theorem~\ref{T mainmain}. A standard approach would involve
verifying directly the uniform integrability of the process $\{M_{n}^{\prime
}(\infty )\}_{n\in \mathbb{N}}$ under $\{\widehat P_{n}^h(\cdot
| \{T_{0}^{n}<T_{N}^{n}\})\}_{n\in \mathbb{N}}$, and the
expectation of the term $\exp \left( T_{0}^{n}/2\right) $ is difficult to
handle. Our technique bypasses the need for this by a simple application of
the strong Markov property as shown in \eqref{eq:10}. \qed
\end{remark}

\setlength{\bibsep}{1pt}
\bibliographystyle{apalike}
\bibliography{aa_bib}
{} {} 
{}

\end{document}